\documentclass[11pt]{article}
\usepackage{arxiv}
\usepackage{amsfonts}
\usepackage{amsmath}
\usepackage{amsthm}
\usepackage{pxfonts}
\usepackage{mathtools}
\usepackage{tocloft}
\usepackage{todonotes}
\usepackage{extarrows}
\usepackage{titlesec}
\usepackage{enumitem}

\usepackage{wrapfig}
\usepackage{caption}
\usepackage{hyperref}
\usepackage{subcaption}

\theoremstyle{definition}
\newtheorem{definition}{Definition}[]

\newtheorem{remark}[definition]{Remark}

\theoremstyle{plain}
\newtheorem{theorem}[definition]{Theorem}

\newtheorem{proposition}[definition]{Proposition}
\newtheorem{lemma}[definition]{Lemma}

\title{Robin Pre-Training for the Deep Ritz Method}

\date{May 2021}

\begin{document}

\author{
    Luca Courte \\
    University of Freiburg \\
    \texttt{luca.courte@mathematik.uni-freiburg.de}
   \And
    Marius Zeinhofer \\
    University of Freiburg\\
    \texttt{marius.zeinhofer@mathematik.uni-freiburg.de}
}

\maketitle

\begin{abstract}
    We compare different training strategies for the Deep Ritz Method for elliptic equations with Dirichlet boundary conditions and highlight the problems arising from the boundary values. We distinguish between an exact resolution of the boundary values by introducing a distance function and the approximation through a Robin Boundary Value problem. However, distance functions are difficult to obtain for complex domains. Therefore, it is more feasible to solve a Robin Boundary Value problem which approximates the solution to the Dirichlet Boundary Value problem, yet the na\"ive approach to this problem becomes unstable for large penalizations. A novel method to compensate this problem is proposed using a small penalization strength to pre-train the model before the main training on the target penalization strength is conducted. We present numerical and theoretical evidence that the proposed method is beneficial.
\end{abstract}

\section{Introduction}
After exceptional success in machine learning, neural networks become increasingly popular in numerical analysis. Strategies to approximate solutions of partial differential equations (PDEs) based on neural networks can be traced back to \cite{dissanayake1994neural} and \cite{lagaris1998artificial}. Their approaches are currently being revived and extended thanks to increased computational power and ease of implementation in frameworks like Tensorflow and PyTorch, see \cite{sirignano2018dgm, weinan2018deep, raissi2018hidden, li2020fourier, abadi2016tensorflow, paszke2019pytorch}.

Presently, a common drawback of neural network based methods is their poor reliability -- a failure to learn even a simple solution to a PDE is not uncommon, see \cite{wang2020understanding, van2020optimally}. In this note, we show that using the boundary penalty method known from the finite element literature, see for instance \cite{babuvska1973finite}, to approximately enforce Dirichlet boundary conditions in the Deep Ritz Method leads to such an unreliability. More precisely, large penalization parameters -- albeit mandatory for an accurate solution -- lead to highly fluctuating errors and sometimes even to a failure to approximate the solution at all. We present a pre-training strategy to alleviate this issue. The key idea is to pre-train the model using a small penalization parameter and subsequently to shift it by an optimal amount before conducting the training with the desired large penalization. This approach is supported by a theoretical motivation which shows that the difference between a (perfectly) pre-trained model and the true solution to the zero boundary value problem can be expanded in a orthonormal basis with the first basis function being constant. The shifting of the pre-trained model then corresponds to the elimination of the first Fourier coefficient.

Although our numerical results are conducted for the Poisson equation on model domains, the method is applicable to general domains and general elliptic equations and can easily be combined with more advanced optimization strategies as for example proposed by \cite{cyr2020robust, lee2021partition}. 
\subsection{Organization}
The paper is organized as follows. After clarifying notation in \ref{sec:Notation}, we give a short introduction to the Deep Ritz Method and related error estimates in \ref{sec:DeepRitzMethod}, then we give an overview of our main results and the numerical experiments in \ref{sec:NumericalMethods}. We present the details and results of the numerical experiments in sections \ref{sec:naive}, \ref{sec:pre-training} and \ref{sec:exact}. Finally, we draw conclusions in section \ref{sec:conclusions}.
\subsection{Notation}\label{sec:Notation}
On an open subset $\Omega$ of $\mathbb{R}^d$ with boundary $\partial\Omega$, we denote the Sobolev spaces of square integrable functions with square integrable weak derivatives by $H^1(\Omega)$. The subspace of $H^1(\Omega)$ with zero boundary values in the trace sense is denoted by $H^1_0(\Omega)$. The reader is referred to \cite{grisvard2011elliptic} for the precise definitions and more information on Sobolev spaces. For multivariate, scalar valued functions, we use the symbol $\Delta$ for the Laplace operator and $\nabla$ for the gradient. For a parameter $\theta$ in a parameter set $\Theta \subseteq \mathbb{R}^N$ we denote by $u_\theta$ the neural network function corresponding to $\theta$.
\subsection{The Deep Ritz Method}\label{sec:DeepRitzMethod}
In this section we briefly recall the Deep Ritz Method and the corresponding error estimates for linear elliptic equations. In general, the Deep Ritz Method proposed by \cite{weinan2018deep}, transforms the variational formulation of a PDE -- if available -- into a finite dimensional optimization problem using neural network type functions as an ansatz class. For example, suppose we want to approximately solve
\begin{equation}\label{eq:Dirichlet}
    \begin{split}
        -\Delta u & = f \quad \text{in } \Omega \\
    u & = 0 \quad \text{on } \partial\Omega,
    \end{split}
\end{equation}
where $\Omega$ is an open and bounded set in $\mathbb{R}^d$. It is well known that for a function \(u\in H^1_0(\Omega)\) and a right-hand side \(f\in H^1_0(\Omega)^*\) it is equivalent to solve \eqref{eq:Dirichlet} and to be a minimizer of the following optimization problem
\begin{equation}\label{eq:DirichletEnergy}
    u\in \underset{v\in H^1_0(\Omega)}{\arg\min} \; \frac12 \int_{\Omega} \left\lvert \nabla v\right\rvert^2 \mathrm dx - f(v).
\end{equation}
Now, we want want to minimize \eqref{eq:DirichletEnergy} over a class of neural network functions. However, it is unfeasible to enforce zero boundary values due to the unconstrained nature of neural networks. A solution is to use the boundary penalty method which allows to approximate \eqref{eq:DirichletEnergy} by an unconstrained problem. This approach was used by \cite{weinan2018deep}. More precisely, let $\lambda > 0$ be a fixed penalization parameter and denote by $\Theta$ a set of neural network parameters and consider the problem of finding a (quasi-)minimiser of the loss function
\begin{equation}\label{eq:LossFunction}
    \mathcal{L}_\lambda: \Theta \to \mathbb{R}, \quad \mathcal{L}_\lambda(\theta) = \frac12 \int_\Omega|\nabla u_\theta|^2\mathrm dx - f(u_\theta) + \lambda \int_{\partial\Omega}u_\theta^2\mathrm ds
\end{equation}
This is now an unconstrained optimisation problem that enforces zero boundary conditions approximately depending on the size of $\lambda$. To gain insight into the nature of this approach, we consider the energy, i.e., the loss function extended to all of $H^1(\Omega)$
\begin{equation*}
    E_\lambda: H^1(\Omega) \to \mathbb{R}, \quad E_\lambda(u) = \frac12 \int_\Omega|\nabla u|^2\mathrm dx - f(u) + \lambda \int_{\partial\Omega}u^2\mathrm ds
\end{equation*}
and its associated Euler Lagrange equation
\begin{equation}\label{eq:Robin}
    \begin{split}
        -\Delta u & = f \quad \text{in } \Omega \\
    \partial_n u + 2\lambda u & = 0 \quad \text{on } \partial\Omega.
    \end{split}
\end{equation}
Thus, using \eqref{eq:LossFunction} to approximate \eqref{eq:Dirichlet} means to use a Robin boundary value problem to approximate a Dirichlet condition. \cite{muller2021error} have shown that under certain regularity conditions, the error introduced by this scheme can be bounded by
\begin{equation}\label{eq:ErrorEstimate}
    \lVert u_\theta - u_0 \rVert_{H^1(\Omega)} \leq \lVert u_\theta - u_\lambda \rVert_{H^1(\Omega)} + \lVert u_\lambda - u_0 \rVert_{H^1(\Omega)} \leq \sqrt{ \frac{2\delta_\lambda}{\alpha_\lambda} + \frac{1}{\alpha_\lambda} \inf_{\tilde\theta\in\Theta}\lVert u_{\tilde\theta}-u_\lambda \rVert_\lambda^2 } + C(\Omega)\cdot\lVert f \rVert_{L^2(\Omega)}\lambda^{-1}.
\end{equation}
Here, $u_0$ denotes the solution to \eqref{eq:Dirichlet}, $u_\lambda$ is the solution to the Robin Boundary Value Problem \eqref{eq:Robin} and $u_\theta$ is a realization of a neural network with parameters $\theta$, e.g., produced through training. The remaining quantities are given by
\begin{gather*}
    \delta_\lambda = \mathcal{L}_\lambda(\theta) - \inf_{\tilde\theta\in\Theta}\mathcal{L}_\lambda(\tilde\theta), \quad \lVert \cdot \rVert_\lambda^2 = \lVert \nabla\cdot \rVert^2 + \lambda \lVert \cdot \rVert^2_{L^2(\partial\Omega)}, \quad \alpha_\lambda = \inf_{\lVert u \rVert_{H^1(\Omega)}=1}\lVert u \rVert_\lambda^2
\end{gather*}
and $C(\Omega)$ is a domain dependent constant. The estimate above indicates five potential sources of errors that guide our interpretation of the numerical experiments.
\begin{itemize}
    \item [(i)] The error introduced by the fact that the Deep Ritz Method approximates a Robin Boundary Value Problem. This is encoded in the term \[C(\Omega)\cdot\lVert f \rVert_{L^2(\Omega)}\lambda^{-1}.\]
    \item[(ii)] The optimization error made by the imperfect training. This is captured by the quantity 
    \begin{equation}\label{eq:ImperfectTrainingError} 
        \frac{2\delta_\lambda}{\alpha_\lambda} = \frac{2}{\alpha_\lambda}\left(\mathcal{L}_\lambda(\theta) - \inf_{\tilde\theta\in\Theta}\mathcal{L}_\lambda(\tilde\theta)\right). 
    \end{equation} 
    \item[(iii)] The error due to the lack of expressiveness of the ansatz class
    \[ \inf_{\tilde\theta\in\Theta}\lVert u_{\tilde\theta}-u_\lambda \rVert_\lambda^2,  \]
    i.e., the distance of an ideal neural network from the Robin solution in the norm $\lVert\cdot\rVert_\lambda$.
    \item[(iv)] The error caused by the numerical approximation of the integrals in the loss function.
    \item[(v)] The error introduced by potentially large constants $\alpha_\lambda^{-1}$ and $C(\Omega)$.
\end{itemize}

\begin{remark}[Sharpness of the Error Estimates] As we employ the triangle inequality in \eqref{eq:ErrorEstimate} it is unclear whether the estimate can be improved and might therefore be of limited use for discussing the errors observed in practice. However, for Dirichlet boundary conditions, we are interested in large values of $\lambda$ and in that case the influence of the triangle inequality is marginal as $u_\lambda$ approaches $u_0$. Furthermore, the estimates of the two summands resulting from the triangle equation's application are optimal. The first uses a C\'ea Lemma with the only estimate being a norm equivalence, compare to Proposition 3.1 in \cite{muller2021error}. The optimality of the second term follows by the observation that the solution to Dirichlet problem \eqref{eq:Dirichlet} on the disk differs by $\mathcal{O}(\lambda^{-1})$ from the solution of the Robin problem on the disk \eqref{eq:Robin}. We therefore deem the estimate \eqref{eq:ErrorEstimate} reliable to discuss the errors observed in practice.
\end{remark}

\subsection{Numerical Methods, Experiments \& Main Results}\label{sec:NumericalMethods}
We investigate and compare three approaches to deal with zero boundary values in the Deep Ritz Method. First, in section \ref{sec:naive} as a na\"ive idea, we use the loss function $\mathcal{L}_\lambda$ defined in \eqref{eq:LossFunction} for a fixed $\lambda > 0$. It turns out that this method becomes increasingly unstable as $\lambda$ increases. Furthermore, we find strong evidence that the most relevant error is due to the imperfect training, as listed in \eqref{eq:ImperfectTrainingError} above.

Second, as the direct usage of $\mathcal{L}_\lambda$ is unreliable, we propose a novel two step scheme. In the first step we train the network for a small parameter $\lambda_P$, i.e., we use the loss function $\mathcal{L}_{\lambda_P}$. We refer to this step as the pre-training. Then, we choose our (large) target value $\lambda_T$, shift the function produced through the pre-training optimally with respect to $\mathcal{L}_{\lambda_T}$ by a constant, and proceed with the training. See section \ref{sec:pre-training} for a theoretical interpretation of this method and numerical results. We find that this approach increases accuracy and reduces the variance of the error over multiple runs.

Finally, we circumvent the penalty method completely by encoding zero boundary values directly into the model as originally proposed by \cite{lagaris1998artificial}. The idea is to use a smooth function $d:\overline{\Omega}\to [0,\infty)$ that satisfies $d(x) = 0$ on $\partial\Omega$ and $d(x) >0$ in $\Omega$. We call $d$ a smooth distance function. Then, the ansatz space for the minimization of \eqref{eq:DirichletEnergy} is chosen to be \[ \{u_\theta\cdot d \mid \theta\in\Theta \}\subset H^1_0(\Omega), \] where $\Theta$ is a parameter set of a neural network architecture. Interestingly, we find that the choice of the smooth distance function has a significant influence on the accuracy of this approach, making an ideal choice difficult without a priori knowledge of the solution. Still, this approach performs best, with lowest error and variance. However, a smooth distance function is difficult or costly to obtain on complex domains. We refer the reader to \cite{berg2018unified}.

We compare these three methods for three different problems in two dimensions
\begin{enumerate}
    \item on the disk, i.e., $\Omega = B_{1}(0)$ and $f \equiv 1$,
    \item on the annulus, i.e., $\Omega = B_{2}(0) \setminus B_{1}(0)$ and $f \equiv 1$,
    \item on the square, i.e., $\Omega = [0, 1]^2$ and $f(x_1,x_2) \coloneqq 8\pi^2\sin(2\pi x_1)\sin(2\pi x_2)$.
\end{enumerate}
We chose these domains as exact solutions are analytically available and can be used to estimate the performance of the different optimization schemes. Indeed, we have
\begin{enumerate}
    \item on the disk, the analytical solution is given by $u^D(x) \coloneqq - \frac{1}{4} |x|^2 + \frac{1}{4}$ for $x\in B_1(0)$,
    \item on the annulus, the analytical solution is given by $u^A(x) \coloneqq - \frac{1}{4} |x|^2 + \frac{3}{4\log(2)} \log(|x|) + \frac{1}{4}$ for $x \in B_{2}(0) \setminus B_1(0)$,
    \item on the square, the analytical solution is given by $u^S(x_1,x_2) \coloneqq \sin(2\pi x_1)\sin(2\pi x_2)$ for $(x_1,x_2)\in[0,1]^2$.
\end{enumerate}

\subsection{Discretization, Network Architecture \& Optimization Routine}
To the best of our knowledge, there is no consensus yet on which network architectures and optimization strategies are favorable for the approximation of PDE solutions. We use small networks with moderate depth to keep the computational cost manageable. Also in view of the discretization of the loss functional $\mathcal{L}_\lambda$ in \eqref{eq:LossFunction} this seems reasonable, compare to remark \ref{remark:Warning}. Our precise choices that are global for all experiments are reported below. 
\begin{enumerate}
    \item We use fully connected feed forward networks with four hidden layers and input dimension two and output dimension one. All hidden layers have 14 neurons. We choose the hyperbolic tangent as our activation function.
    \item We initialize the weights using Glorot uniform initialization and the biases are initialized by zero.
    \item The numerical integration of the integrals for the training is done using fixed evaluation points. For a number $N\in\mathbb{N}$ we use points in the lattice $\frac1N\mathbb{Z}^2$ to approximate the integral of a function $f:\Omega\subset\mathbb{R}^2\to\mathbb{R}$ via 
    \[
        \int_{\Omega}f\mathrm dx = \frac1N \sum_{x_i \in \frac1N\mathbb{Z}^2\cap \operatorname{int}(\Omega)}f(x_i).
    \]
    Integrals over the boundary of a domain are approximated using arclength parametrization with $N\cdot|\partial\Omega|$ many equi-spaced evaluation points. Here $|\partial\Omega|$ denotes the length of $\partial\Omega$. For the square and the annulus we use roughly $250,000$ evaluation points, i.e., the lattice constants are $N=500$ and $N=160$ respectively. For the disk we use roughly $90000$ evaluation points, i.e, $N = 160$. We refer the reader to remark \ref{remark:Warning} and Appendix \ref{appendix:IntegralDiscretization} for a legitimisation of these choices.  
    \item For the optimization process we use the Adam optimizer with $10000$ iterations. Depending on the experiment we vary slightly with the learning rate. Details can be found in the corresponding sections.
    \item The relative $L^2(\Omega)$ and $H^1(\Omega)$ errors are computed using $10^6$ uniformly sampled evaluation points in $\Omega$ for every integral appearing in the respective norms.
\end{enumerate}
Note that we do not employ a stochastic gradient descent. In our loss function, we do not introduce a batch size and in the approximation of the integrals in \eqref{eq:LossFunction} we do not use random points. The only random effect is introduced through the networks weight initialization. 

In the following, it will sometimes be important to distinguish between the continuous loss as stated in \eqref{eq:LossFunction} and its numerical approximation which we call numerical loss function. It is given by applying the integral discretization described above to \eqref{eq:LossFunction}. More precisely, denoting by $N_{\textrm{int}}$ and $N_{\textrm{bdr}}$ the number of evaluation points $(x_i)$ in the interior of $\Omega$ and $(z_j)$ on its  boundary $\partial\Omega$ we have
\begin{equation}\label{eq:numericalLoss}
     \mathcal{L}^N_\lambda\colon\Theta\to\mathbb{R}, \quad \theta\mapsto \frac{1}{N_{\textrm{int}}} \sum_{i = 1}^{N_{\textrm{int}}}\left[ \frac12|\nabla  u_\theta(x_i)|^2 - f(x_i)u_\theta(x_i)\right] + \frac{\lambda}{N_{\textrm{bdr}}} \sum_{j = 1}^{N_{\textrm{bdr}}} u_\theta^2(z_j).
\end{equation}
\begin{remark}[A Warning]\label{remark:Warning} The numerical loss function is severely ill-posed on expressive sets of functions. For arbitary $\lambda$ consider 
\begin{equation*}
    \mathcal{L}^N_\lambda: H^1(\Omega)\cap C^1(\overline{\Omega}) \to \mathbb{R} \quad \text{with}\quad u \mapsto \mathcal{L}^N_\lambda(u).
\end{equation*}
Then $\mathcal{L}_\lambda^N$ is unbounded from below (and above) and does not admit a global minimizer. An arbitrary small value of $\mathcal{L}_\lambda^N$ can be achieved by a function that is zero on $\partial\Omega$, has vanishing gradients on the evaluation points in the interior and takes large values on the interior evaluation points. However, for a fixed size neural network this effect can be circumvented, provided enough evaluation points are used. This is another motivation to use small to medium size neural networks when working with the Deep Ritz Method in our setting.
\end{remark}

\section{Na\"ive Method}\label{sec:naive}
We apply the Deep Ritz Method with boundary penalty to the equations listed in section \ref{sec:NumericalMethods}. We use different but fixed penalization strengths and observe that large penalization values make the method unstable. We conclude that the decisive error contribution stems from the optimization.
\subsection*{Experiment Setup}
For the three equations, we use the boundary penalization strengths $\lambda = 1, 5, 10, 50, 100, 500, 1000, 5000$ and $10000$ and every experiment is conducted $25$ times to capture the stochastic influence of the initialization. The optimization process is carried out with $10000$ iterations of the Adam optimizer with learning rate $0.001$. The remaining technical details such as network architecture and integral discretization are recorded in section \ref{sec:NumericalMethods} and do not differ across the different experiments in this note. We record the relative $L^2(\Omega)$ and $H^1(\Omega)$ errors. Furthermore, in the case of the disk and the annulus also the solutions $u_D^\lambda$ and $u_A^\lambda$ to the Robin problem \eqref{eq:Robin} have an analytical form
\begin{gather*}
    u_D^\lambda(x) = -\frac14(x_1^2+x_2^2) + \frac14 +\frac{1}{4\lambda}
    \quad\text{and}\quad
    u_A^\lambda(x) = -\frac{\lvert x \rvert^2}{4} + C_1\log(\lvert x \rvert) + C_2,
\end{gather*}
where \[C_1 = \frac{1 + 2\lambda\left(\frac34 + \frac{1}{4\lambda}\right)}{\frac12 + 2\lambda\left(\log(2) + \frac{1}{2\lambda} \right)}\quad\text{and}\quad C_2 = \frac{C_1}{2\lambda} + \frac14 - \frac{1}{4\lambda}.\]
so we keep track of the relative $L^2(\Omega)$ and $H^1(\Omega)$ errors with respect to $u_D^\lambda$ and $u_A^\lambda$ as well.
\subsection*{Discussion}
\begin{figure}
    \centering
    \begin{subfigure}{0.48\linewidth}
    \includegraphics[width=\linewidth]{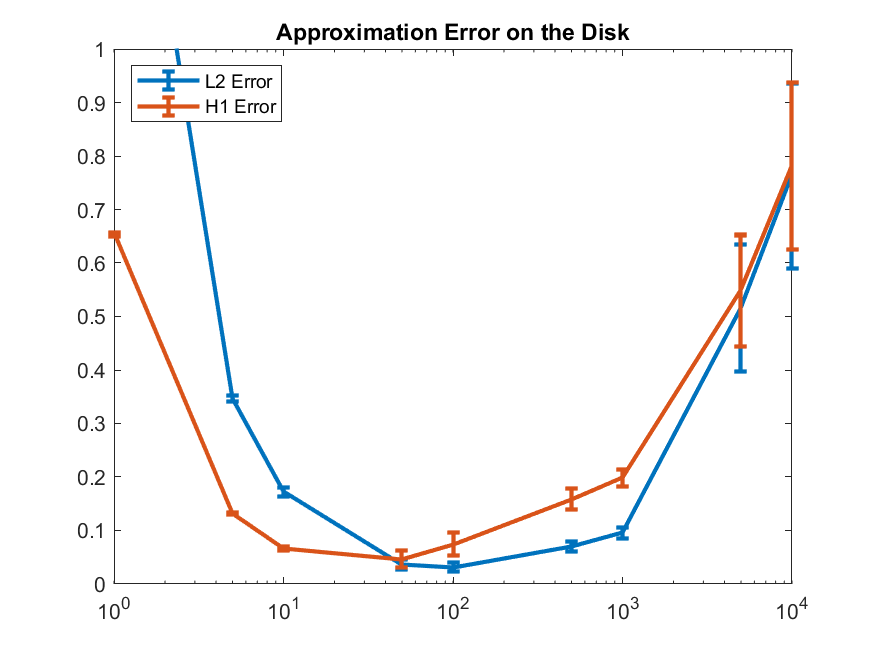}
    \end{subfigure}
    \begin{subfigure}{0.48\linewidth}
    \includegraphics[width=\linewidth]{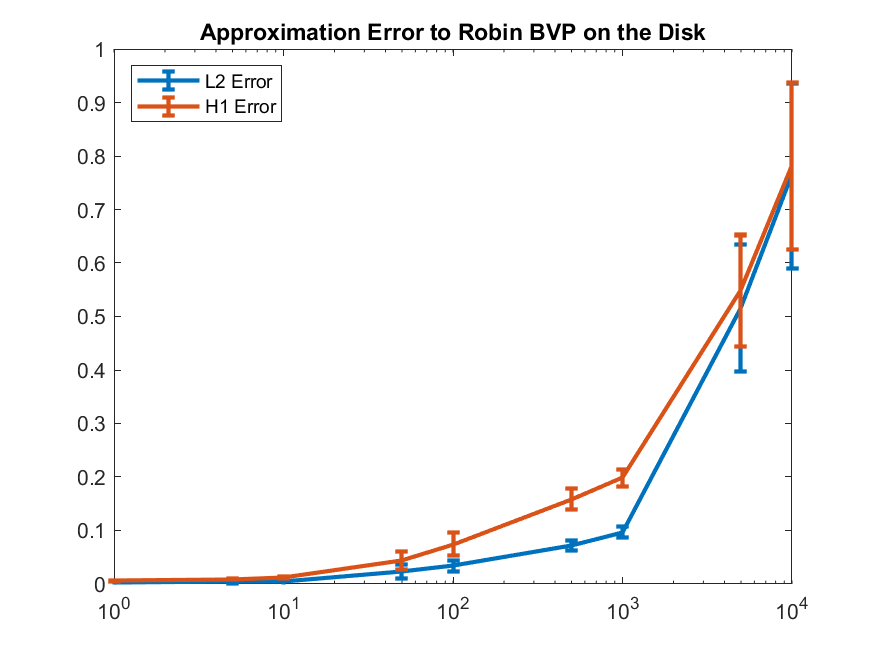}
    \end{subfigure}
    \caption{The expected relative error with its sample standard deviation for different values of $\lambda$ on the disk with respect to the analytical solution of the Dirichlet BVP (on the left) and to the solution of the Robin BVP (on the right).}
    \label{fig:reference_disk}
\end{figure}
\begin{figure}
    \centering
    \begin{subfigure}{0.48\linewidth}
    \includegraphics[width=\linewidth]{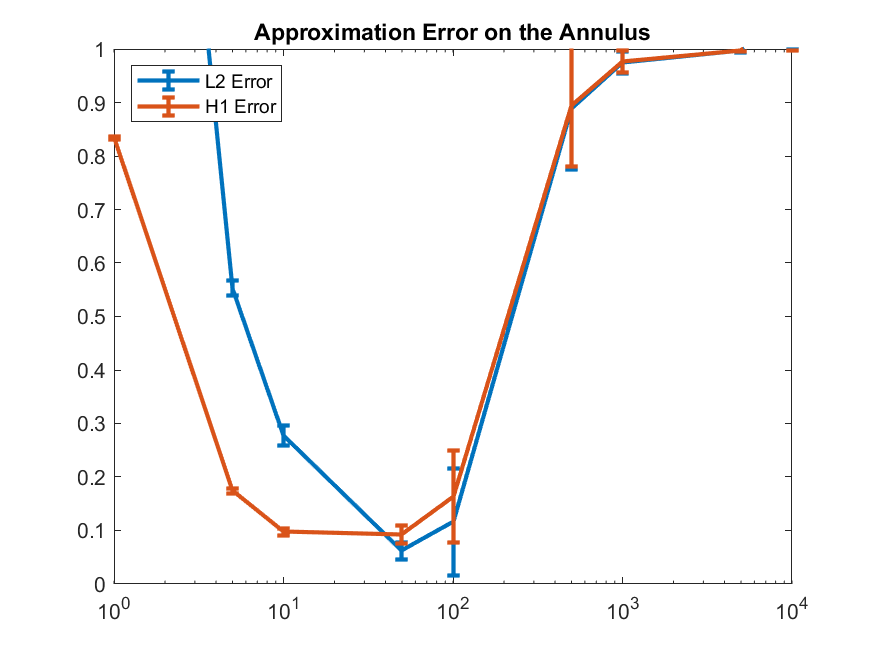}
    \end{subfigure}
    \begin{subfigure}{0.48\linewidth}
    \includegraphics[width=\linewidth]{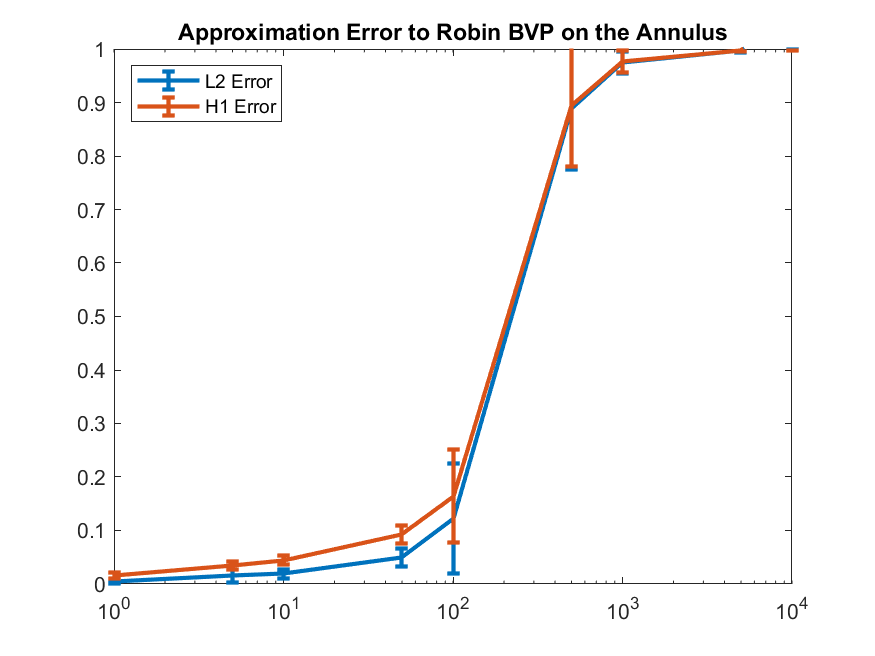}
    \end{subfigure}
    \caption{The expected relative error with its sample standard deviation  for different values of $\lambda$ on the annulus with respect to the analytical solution of the Dirichlet BVP (on the left) and to the solution of the Robin BVP (on the right).}
    \label{fig:reference_annulus}
\end{figure}
\begin{figure}
    \centering
    \begin{subfigure}{0.48\linewidth}
    \includegraphics[width=\linewidth]{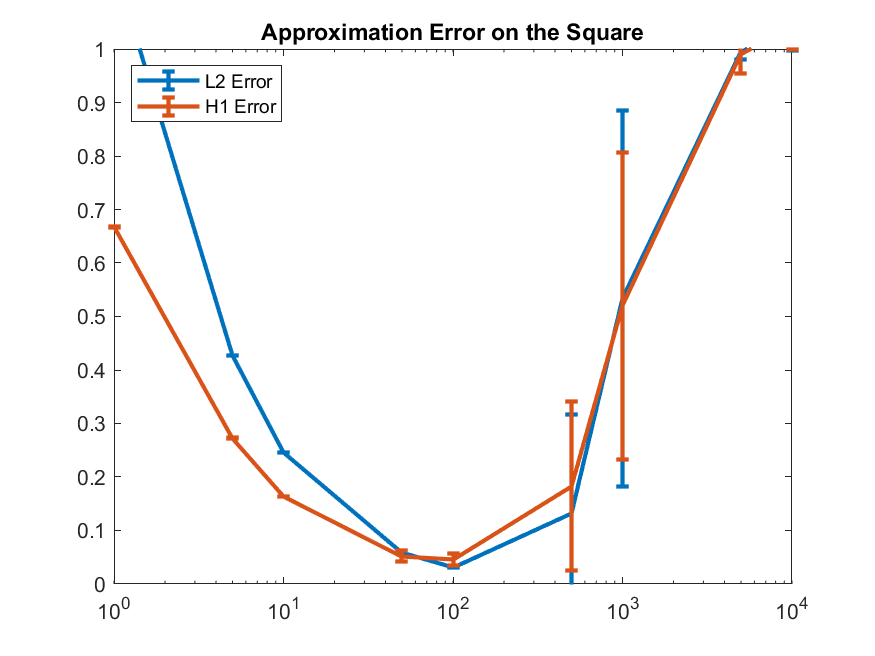}
    \end{subfigure}
    \caption{The expected relative error with its sample standard deviation  for different values of $\lambda$ on the square.}
    \label{fig:reference_square}
\end{figure}

We report the relative errors in figure \ref{fig:reference_disk}, figure \ref{fig:reference_annulus}, and figure \ref{fig:reference_square} respectively. We observe a clear trade-off between too large and too small penalization parameters for all three equations. This is in accordance with the error estimates \eqref{eq:ErrorEstimate}.
We begin by investigating the experiment on the disk more closely. Specializing the error estimates to the case of the disk and estimating all unknown constants yields
\begin{equation*}
    \lVert u_\theta - u \rVert_{H^1(\Omega)} \leq \sqrt{ \frac{8+10\lambda}{\lambda}\delta_\lambda + \frac{4+5\lambda}{\lambda} \inf_{\tilde\theta\in\Theta}\lVert u_{\tilde\theta}-u_\lambda \rVert_\lambda^2 } + \frac{\pi}{4\lambda}.
\end{equation*}
We refer the reader to Appendix \ref{appendix:ErrorEstimateDisk} for the derivation of this estimate. The trade-off is due to the fact that the second term vanishes with $\lambda\to\infty$ but is large for small values of $\lambda$ and in the first term both
\begin{gather*}
    \delta_\lambda 
    = 
    \mathcal{L}_\lambda(\theta) - \inf_{\tilde\theta\in\Theta}\mathcal{L}_\lambda(\tilde\theta) \quad\text{and}\quad \inf_{\tilde\theta\in\Theta}\lVert u_{\tilde\theta}-u_\lambda \rVert_\lambda^2 
    =
    \inf_{\tilde\theta\in\Theta}\left[ \int_\Omega \lvert \nabla (u_{\tilde\theta} - u_\lambda) \rvert^2\mathrm dx + \lambda\int_{\partial\Omega}(u_{\tilde\theta} - u_\lambda)^2\mathrm ds\right]
\end{gather*}
can potentially be large for large values of $\lambda$. From the error plot in figure \ref{fig:reference_disk} (left) it is not immediately clear whether $\delta_\lambda$ or $\inf_{\tilde\theta\in\Theta}\lVert u_{\tilde\theta}-u_\lambda \rVert_\lambda^2$ introduces larger errors. However, a look at the solution to the Robin problem on the disk reveals that $u_D^\lambda$ can be approximated -- at least pointwise -- in the same quality for every $\lambda$ as different values of $\lambda$ only correspond to different constant terms in $u_D^\lambda$ which can be taken care of by the last bias of the approximating neural network. Still, the error to the Robin solution increases by a factor of over $100$ with growing penalization strength and also the variance of the error over multiple runs increases drastically, see figure \ref{fig:reference_disk} (right).

\begin{table}[ht]
\centering
\caption{Best relative errors with sample standard deviations}
\label{table:BestErrorsNaive}
 \begin{tabular}{| c | c | c |} 
 \hline
 Domain & Best relative $L^2$ error & Best relative $H^1$ error  \\ [0.5ex] 
 \hline\hline
 Disk & $(3.14 \pm 0.82) \cdot 10^{-2}$ achieved for $\lambda = 100$ & $(4.61\pm 1.61) \cdot 10^{-2}$ achieved for $\lambda = 50$  \\ 
 \hline
 Annulus & $(6.19 \pm 1.55) \cdot 10^{-2}$ achieved for $\lambda = 50$ & $(9.33\pm 1.68) \cdot 10^{-2}$ achieved for $\lambda = 50$ \\ 
 \hline
 Square & $(3.17 \pm 0.12) \cdot 10^{-2}$ achieved for $\lambda = 100$ & $(4.62\pm 1.14) \cdot 10^{-2}$ achieved for $\lambda = 100$ \\ 
 \hline
 
\end{tabular}
\end{table}

A similar behavior is observed for the annulus, with high accuracy and low variance when measuring how well the Deep Ritz Method approximates the Robin solution for small penalization strengths. We refer to figure \ref{fig:reference_annulus}. In case of the square we do not have access to the Robin solution, yet an increasingly unstable training process is also clearly visible here, see figure \ref{fig:reference_square}. In particular, for all three equations a total failure to train, i.e., a relative error of around $100\%$ frequently occurs in the large penalization regime. In this case the zero function is learned. Heuristically, we attribute the deteriorating accuracy to a growing number of increasingly attractive, poor local minima in the loss landscape for large penalization strengths until for very large penalizations even the zero function can result from training. A large value of $\lambda$ forces the gradient descent dynamics violently towards zero boundary conditions at the expense of other features of the PDEs solution, hence good minima become less attractive and difficult to find. For the experiments, the overall best performing $\lambda$ are reported in table \ref{table:BestErrorsNaive}.

In our discussion so far, we did not take the numerical approximation of the integrals into account as the error estimates we used only hold for the continuous loss. As noted in remark \ref{remark:Warning} this is a potential source for errors. To exclude a relevant influence, we use a typical training session of one of the equations in section \ref{sec:NumericalMethods} and we monitor the loss for two different approximations of the integrals; one using the standard, medium fine grid and the other considerably finer. We conclude that the error introduced by the integral approximation can be neglected and refer to Appendix \ref{appendix:IntegralDiscretization} for the details.
\section{Refined Method}\label{sec:pre-training}
In the previous experiment, we have seen that the na\"ive approach using the boundary penalty method introduces large errors and unstable training dynamics, especially when compared to the essential boundary conditions of the Robin problem, i.e, the case when $\lambda$ is small. To mitigate this effect, we propose to conduct a pre-training using a low penalization strength. The motivation behind this strategy is a similarity of the solution $u_\lambda$ to the Robin problem \eqref{eq:Robin} and the solution $u_0$ to the Dirichlet problem \eqref{eq:Dirichlet}. We find that the proposed pre-training significantly reduces both the errors and the errors' variance.
\subsection{Experiment Setup}
We consider the equations listed in section \ref{sec:NumericalMethods}. The following training strategy is used
\begin{enumerate}
    \item [(i)] We set a pre-training penalization strength, in our case $\lambda_P = 1$, and train the model using the loss function $\mathcal{L}_{\lambda_P}$ for $4000$ iterations with the Adam optimizer. Here the first $1000$ iterations the learning rate is set to $0.01$ and for the remaining $3000$ iterations we use $0.001$. Let us denote the the neural network realization that results from this training by $u_{\theta_P}$.
    \item[(ii)] As target penalization strengths we use $\lambda_T = 1,\ 5,\ 10, \ 50,\ 100, \ 500, \ 1000, \ 5000, \ 10000$, and shift the network function $u_{\theta_P}$ adding a optimal constant $t_{\lambda_T}$ namely
    \[ \tilde{u}_{\theta_P} = u_{\theta_P} + t_{\lambda_T} = u_{\theta_P} + \frac{1}{2\lambda_T|\partial \Omega|} \int_\Omega f(x) \;\mathrm{d}x - \frac{1}{|\partial \Omega|} \int_{\partial \Omega} u_{\theta_P}(x) \;\mathrm{d}x. \]
    We denote the shifted function by $\tilde{u}_{\theta_P}$. Finally we train $\tilde{u}_{\theta_P}$ for $6000$ iterations with the Adam optimizer with learning rate $0.001$ and the loss function $\mathcal{L}_{\lambda_T}$ to conclude the training.
\end{enumerate}
A theoretical motivation for this approach is given in section \ref{sec:Motivationpre-training}.
\subsection{Discussion}

\begin{figure}
    \centering
    \begin{subfigure}{0.48\linewidth}
    \includegraphics[width=\linewidth]{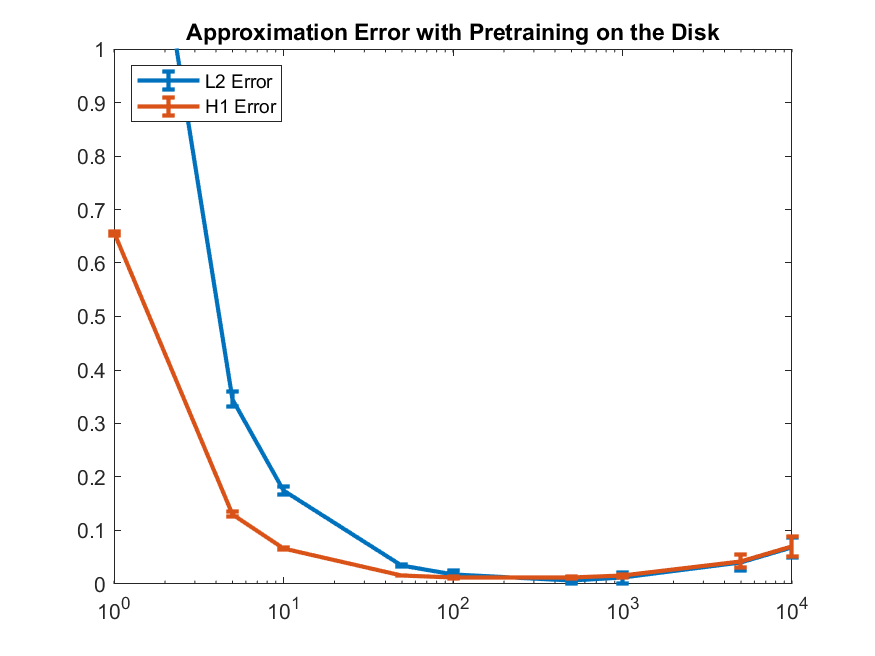}
    \end{subfigure}
    \begin{subfigure}{0.48\linewidth}
    \includegraphics[width=\linewidth]{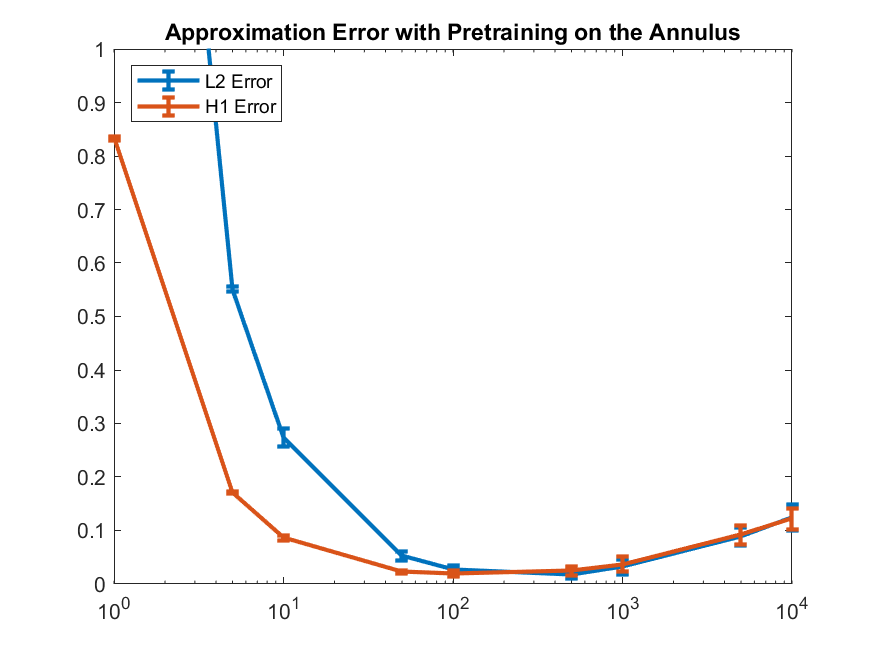}
    \end{subfigure}
    \begin{subfigure}{0.48\linewidth}
    \includegraphics[width=\linewidth]{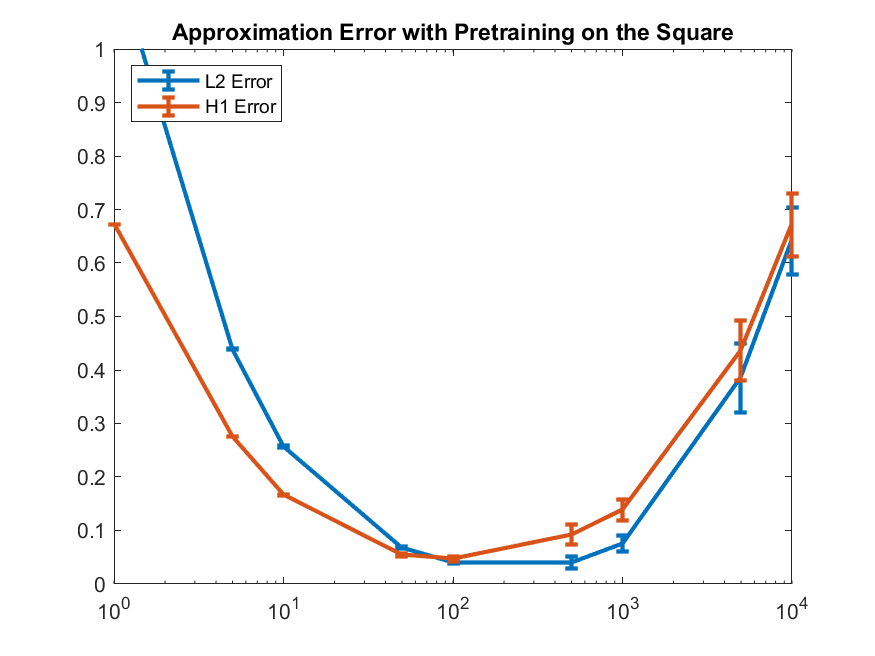}
    \end{subfigure}
    \caption{The expected relative error with its sample standard deviation for different values of $\lambda$ with pre-training.}
    \label{fig:pre-training}
\end{figure}

\begin{table}[ht]
\centering
\caption{Best relative errors with sample standard deviations with pre-training}
\label{table:BestErrorspre-training}
 \begin{tabular}{| c | c | c |} 
 \hline
 Domain & Best relative $L^2$ error & Best relative $H^1$ error  \\ [0.5ex] 
 \hline\hline
 Disk & $(0.63 \pm 0.48) \cdot 10^{-2}$ achieved for $\lambda = 500$ & $(1.17\pm 0.21) \cdot 10^{-2}$ achieved for $\lambda = 500$  \\ 
 \hline
 Annulus & $(1.81 \pm 0.80) \cdot 10^{-2}$ achieved for $\lambda = 500$ & $(1.97\pm 0.47) \cdot 10^{-2}$ achieved for $\lambda = 100$ \\ 
 \hline
 Square & $(3.05 \pm 0.21) \cdot 10^{-2}$ achieved for $\lambda = 100$ & $(4.54\pm 0.48) \cdot 10^{-2}$ achieved for $\lambda = 100$ \\ 
 \hline
 
\end{tabular}
\end{table}

We report the relative $L^2(\Omega)$ and $H^1(\Omega)$ errors obtained through using the proposed pre-training strategy in figure \ref{fig:pre-training} and the best results we obtained in table \ref{table:BestErrorspre-training}. We clearly see beneficial effects for both the accuracy and the variance of the errors for all three equations. In case of the disk, the relative error is almost reduced by a full magnitude. Also on the annulus, we see a drastic improvement of accuracy. For the square, the improvement in accuracy is smaller, however, the variance in the reported errors is decreased drastically as well. This is no surprise as the example of the square possesses an oscillating solution in the interior and we believe that the main error is due to this comparatively complicated solution. In the following, we list the advantages of our proposed pre-training strategy.
\begin{enumerate}
    \item [(i)] Improved accuracy and increased reliability as discussed above.
    \item[(ii)] Potential computational savings. Even though our method requires to pre-train the model we believe that the overall computational cost can be reduced. In fact, using larger learning rates is possible as the loss landscape is less rough for small values of $\lambda$. We have already exploited this in our experiment setup. Also in the main training, potentially larger learning rates can be used as the model is already in a reasonable state. We leave the fine-tuning of learning rate schedules for future research.
    \item[(iii)] Decreased sensitivity to the choice of penalization parameter. Particularly, the experiments on the disk and the annulus suggest that the accuracy of the method does not depend as strongly on the penalization strength as in the case without pre-training. This facilitates choosing a reasonable penalization strength.
\end{enumerate}
\subsection{Theoretical Motivation for Pre-Training}\label{sec:Motivationpre-training}
In this section we justify our pre-training and especially the shift by a constant value. Remember that the boundary penalty method with penalty parameter $\lambda$ corresponds to a Robin boundary value problem as explained in \eqref{eq:Robin}. For the convenience of the reader we repeat the equation  
\begin{equation*}
    \begin{split}
        -\Delta u & = f \quad \text{in } \Omega \\
        \partial_nu + 2\lambda u & = 0 \quad \text{on } \partial\Omega,
    \end{split}
\end{equation*}
and we denote its solution by $u_\lambda$. Heuristically, letting $\lambda\to\infty$ we recover the Dirichlet problem
\begin{equation*}
    \begin{split}
        -\Delta u & = f \quad \text{in } \Omega \\
        u & = 0 \quad \text{on } \partial\Omega,
    \end{split}
\end{equation*}
with solution $u_0$. Our motivation for optimizing among constants stems from the fact that the difference $v_\lambda \coloneqq u_0 - u_\lambda$ can be expanded in a suitable orthonormal basis $(e_j)_{j\in\mathbb{N}}$ -- a Steklov eigenbasis -- and is given by
\begin{equation}\label{eq:AbstractEigenExpansion}
    v_\lambda = \sum_{j=0}^\infty c_j e_j,
\end{equation}
where in fact the leading basis function is constant. Thus, our suggested shifting by a constant eliminates this leading term and brings the result of the pre-training closer to the desired solution of the Dirichlet problem. The details of the expansion \eqref{eq:AbstractEigenExpansion} will be provided in the appendix. 
To determine the optimal constant by which a pre-trained model should be shifted, a simple analytical argument suffices.
\begin{proposition}
    Let $\lambda > 0$ and $L: \mathbb{R}^n \times \mathbb{R}^n \to \mathbb{R}$ a Lagrangian and $u_0:\mathbb{R}^n \to \mathbb{R}$ an admissible function, then the energy
    \[
        E_{\lambda}[u] \coloneqq \int_\Omega L(x, \nabla u(x)) \;\mathrm{d}x - \int_\Omega f(x) u(x) \;\mathrm{d}x + \lambda \int_{\partial \Omega} |u(x)|^2 \;\mathrm{d}x
    \]
    is minimized among all translations $u_t \coloneqq u_0 + t$ for $t\in \mathbb{R}$ if
    \[
        t = \tilde{t} \coloneqq \frac{1}{2\lambda|\partial \Omega|} \int_\Omega f(x) \;\mathrm{d}x - \frac{1}{|\partial \Omega|} \int_{\partial \Omega} u(x) \;\mathrm{d}x.
    \]
    
    
\end{proposition}

\begin{proof}
    The translation $u_t$ for $t\in \mathbb{R}$ that minimizes the energy $E_{\lambda}$ has to satisfy
    \[
        \frac{\mathrm{d}}{\mathrm{d}t} E_{\lambda}[u + t] = 0.
    \]
    Hence, it has to hold
    \[
        0 = -\int_{\Omega} f(x) u(x) \;\mathrm{d}x + 2\lambda \int_{\partial \Omega} u(x) + t\;\mathrm{d}x
    \]
    which, after rearranging, implies the stated equation for the translation $\tilde{t} \in \mathbb{R}$. Due to the structure of the energy this has to be a minimizer.
    
\end{proof}
\begin{remark}
The above motivation is not limited to the Laplace operator. In fact, also for the class of self-adjoint elliptic equations the Steklov theory holds and an analogue argument can be made. Note also, that the derivation of the optimal shifting applies to more general elliptic operators as well.
\end{remark}

\section{Exact Resolution of Boundary Values}\label{sec:exact}
We encode the zero boundary conditions directly into the network architecture using a smooth distance function to the boundary as described in section \ref{sec:NumericalMethods}. For every domain we use two different distance functions and observe that the choice of this function has a considerable influence on the accuracy of the method. We conclude that this approach has the potential to be the most accurate but introduces a possibly hard to obtain and difficult to choose hyperparameter -- the distance function.
\subsection{Experiment Setup}
We consider again the three equations introduced in section \ref{sec:NumericalMethods}. For each domain we use two different smooth distance functions. One is based on trigonometric functions and the other is polynomial. Especially for the experiments on the disk and the annulus the trigonometric functions are chosen to prevent a too strong similarity with the solutions of the boundary value problems. More precisely we use
\begin{gather*}
    d^D_{\text{trig}}(r) = \cos\left(\frac{\pi r}{2}\right)\quad\text{and}\quad d^D_{\text{pol}}(r) = r^2 - 1
\end{gather*}
for the disk,
\begin{gather*}
    d^S_{\text{trig}}(x,y) = \sin(\pi x)\sin(\pi y)\quad\text{and}\quad d^S_{\text{pol}}(x,y) = x(x-1)y(y-1)
\end{gather*}
for the square and
\begin{gather*}
    d^A_{\text{trig}}(r) = \sin(\pi r)\quad\text{and}\quad d^A_{\text{pol}}(r) = -(r - 1)(r - 2)
\end{gather*}
for the annulus, with $r = \sqrt{x^2+y^2}$. The optimization process is carried out using $10000$ iterations of the Adam optimizer with learning rate $0.001$ and every experiment (i.e., every domain-distance function pair) is repeated $25$ times to capture the effect of the random initialization. The remaining hyperparameters are chosen as in the other experiments, see section \ref{sec:NumericalMethods} for an overview.

\subsection{Discussion}
We report the relative $L^2(\Omega)$ and relative $H^1(\Omega)$ errors in table \ref{table:DistanceFunctions}. Comparing with the results of the previous experiments, see table \ref{table:BestErrorsNaive} and table \ref{table:BestErrorspre-training}, we see that for both choices of distance functions this method is the most accurate for all three domains. However, the choice of the specific distance function has a drastic influence on the accuracy. In the case of the annulus the polynomial function performs by half a magnitude better and in the case of the disk by a factor of roughly three. We believe this phenomenon is due to the similarity of the polynomial distance function to the solutions of the PDEs. In the case of the disk, the solution is polynomial and in the case of the annulus, it is a sum of a polynomial and a logarithm, see section \ref{sec:NumericalMethods}.

Finally, the conducted experiments with exact boundary conditions suggest that it is reasonable to use such a construction when available. However, the choice of the distance function has a significant influence on the accuracy of the method. Hence, without a priori knowledge of the solution making a good choice remains an open problem. Furthermore, on complex domains it is difficult and computationally expensive to produce such a distance function and can in general only be done approximately. We refer to \cite{berg2018unified}.
\begin{table}[ht]
\centering
\caption{Relative Errors for different choices of distance functions}
\label{table:DistanceFunctions}
 \begin{tabular}{| c | c | c | c | c |} 
 \hline
 Domain & $L^2$ error with $d_{\text{trig}}$ & $H^1$ error with $d_{\text{trig}} $ & $L^2$ error with $d_{\text{pol}}$ & $H^1$ error with $d_{\text{pol}}$  \\ [0.5ex] 
 \hline\hline
 Disk & $(0.37 \pm0.58) \cdot 10^{-2}$ & $(0.7\pm 0.55) \cdot 10^{-2}$ & $(0.14 \pm 0.17)\cdot 10^{-2}$ & $(0.21\pm 0.18) \cdot 10^{-2} $ \\ 
 \hline
 Annulus & $(0.92 \pm 1.15) \cdot 10^{-2}$ & $(1.67\pm 0.97) \cdot 10^{-2}$ & $(0.18 \pm 0.13)\cdot 10^{-2}$ & $(0.29\pm 0.11) \cdot 10^{-2} $ \\ 
 \hline
 Square & $(1.09 \pm 1.1) \cdot 10^{-2}$ & $(1.4\pm 0.4) \cdot 10^{-2}$ & $(0.64 \pm 0.04)\cdot 10^{-2}$ & $(1.78\pm 0.29) \cdot 10^{-2} $ \\ 
 \hline
 
\end{tabular}
\end{table}

\section{Conclusions}\label{sec:conclusions}
We discussed three different approaches to practically realize homogeneous Dirichlet boundary conditions in the Deep Ritz Method for elliptic equations. First, the na\"ive ansatz of using the boundary penalty method for a fixed penalization strength. Second, we proposed a novel strategy where the model is pre-trained using a small penalization parameter before conducting the main training for the target penalization. Finally, we enforced the boundary conditions directly into the ansatz functions by using a smooth distance function. 

We found that the na\"ive approach suffers severely from the instability introduced by large penalization parameters which complicate the training process. This leads to imprecise approximation of the solution and a high variance of the errors. To alleviate this problem, the use of a smooth distance function seems beneficial, especially if one can encode intuition of the solution into the design of the distance function. However, this is seldom possible for complicated domains. In this case, the proposed pre-training strategy can be used without extra computational or algorithmic cost. It is shown that this approach improves both accuracy and reduces the variance of the errors.

Future research directions are the fine-tuning of the proposed pre-training method, such as employing learning rate schedules and the overall reduction of computational cost. Futhermore, we propose to investigate how the method performs on more complex problems, especially paired with more elaborate training algorithms, for example the ones proposed in \cite{cyr2020robust}. 
\section*{Acknowledgement}
LC gratefully acknowledges support from the Fonds National de la Recherche, Luxembourg (AFR Grant 13502370). MZ gratefully acknowledges support from BMBF within the e:Med program in the SyMBoD consortium (grant number 01ZX1910C). The authors thank Johannes M\"uller for valuable discussions and suggestions.

\section*{Declaration of Interest}
Declarations of interest: none.

\bibliographystyle{apalike}
\bibliography{references}
\appendix
\section{Derivation of Error Estimates on the Disk}\label{appendix:ErrorEstimateDisk}
We provide the details for the specialized error estimates on the disk on which we base our discussion in section \ref{sec:naive}.
\begin{lemma}
Let $\Omega = B_1(0)\subset\mathbb{R}^2$ be the unit disk. Fix $\lambda>0$ and let $\Theta$ denote a parameter set of some neural network architecture with $u_\theta\in H^1(\Omega)$ for all $\theta\in\Theta$. We denote by $u_0$ the solution to \eqref{eq:Dirichlet}, by $u_\lambda$ the solution to \eqref{eq:Robin} and we let $\theta$ be arbitrary. Then it holds
\begin{equation*}
    \lVert u_\theta - u_0 \rVert_{H^1(\Omega)} \leq \sqrt{ \frac{8+10\lambda}{\lambda}\delta_\lambda + \frac{4+5\lambda}{\lambda}\inf_{\tilde\theta\in\Theta}\lVert u_{\tilde\theta}-u_\lambda \rVert_\lambda^2 } + \frac{\pi}{4\lambda}
\end{equation*}
where, using $\mathcal{L}_\lambda$ as in \eqref{eq:LossFunction}, we have
\begin{gather*}
    \delta_\lambda = \mathcal{L}_\lambda(\theta) - \inf_{\tilde\theta\in\Theta}\mathcal{L}_\lambda(\tilde\theta)\quad\text{and}\quad \lVert \cdot \rVert_\lambda^2 = \lVert \nabla\cdot\rVert_{L^2(\Omega)}^2 + \lambda \lVert \cdot \rVert_{L^2(\partial\Omega)}^2
\end{gather*}
\end{lemma}
\begin{proof}
    The error estimates derived by \cite{muller2021error} yield
    \begin{equation*}
    \lVert u_\theta - u_0 \rVert_{H^1(\Omega)} 
    \leq
    \sqrt{ \frac{2\delta_\lambda}{\alpha_\lambda} + \frac{1}{\alpha_\lambda} \inf_{\tilde\theta\in\Theta}\lVert u_{\tilde\theta}-u_\lambda \rVert_\lambda^2 } + \lVert u_\lambda - u_0 \rVert_{H^1(\Omega)}.
\end{equation*}
Inspecting the solution formulas for $u_0$ and $u_\lambda$ in section \ref{sec:NumericalMethods} reveals that $u_\lambda - u_0\equiv \frac{1}{4\lambda}$, hence
\[ \lVert u_\lambda - u_0 \rVert_{H^1(\Omega)} = \frac{\pi}{4\lambda}. \]
It remains to estimate $\alpha_\lambda$, the coercivity constant of the bilinear form
\begin{equation*}
    a_\lambda:H^1(\Omega)\times H^1(\Omega)\to\mathbb{R},\quad a_\lambda(u,v) = \int_\Omega\nabla u\nabla v\mathrm dx + \lambda\int_{\partial\Omega}uv\mathrm ds.
\end{equation*}
In fact, $\alpha_\lambda$ may be estimated using the Friedrich's constant $C_{F,\lambda}$ \[\alpha_\lambda \geq \left( C_{F,\lambda} +1 \right)^{-1},\] where $C_{F,\lambda}$ is the smallest constant satisfying
\begin{equation*}
    \int_\Omega u^2\mathrm dx \leq C_{F,\lambda}\left( \int_\Omega |\nabla u|^2\mathrm dx + \lambda \int_{\partial\Omega}u^2\mathrm ds \right)\quad \text{for all }u\in H^1(\Omega).
\end{equation*}
Furthermore, $C_{F,\lambda}$ can be connected to the first eigenvalue of the Robin Laplacian which we denote by $\mu_1^\lambda$. More precisely it holds via the Rayleigh quotient characterization
\begin{equation*}
    \mu_1^\lambda = \min_{u\in H^1(\Omega)\setminus\{0\}} \frac{\int_\Omega|\nabla u|^2\mathrm dx + \lambda \int_{\partial\Omega}u^2\mathrm ds}{\int_\Omega u^2\mathrm dx} = C_{F,\lambda}^{-1}.
\end{equation*}
Finally, in the case of convex domains, two-sided estimates for $\mu_1^\lambda$ are known, see \cite{kovavrik2014lowest}. In particular we get for the disk (as a general convex domain)
\[ \mu_1^\lambda \geq \frac{\lambda}{4+4\lambda} \]
which concludes the proof.
\end{proof}
\section{Steklov Eigenvalue Theory}
We will now provide the details of the expansion \eqref{eq:AbstractEigenExpansion} and repeat briefly the prerequisites from the Steklov eigenvalue theory. For the proofs of our basic results we refer to \cite{muller2021error} and for a more advanced and comprehensive introduction to the Steklov Eigenvalue Theory the reader is referred to \cite{auchmuty2005steklov}. The basic definition in case of the Laplacian is the following.  
\begin{definition}
Let $\Omega\subset\mathbb{R}^d$ be open and bounded. The Steklov eigenvalue problem for the Laplacian consists of finding $(\mu,w)\in \mathbb{R}\times H^1(\Omega)\setminus\{0\}$ such that
\begin{equation}
    \int_\Omega \nabla w\nabla\varphi\mathrm dx = \mu \int_{\partial\Omega} w \varphi \mathrm ds \quad \text{for all }\varphi\in H^1(\Omega).
\end{equation}
We call \(\mu\) a \emph{Steklov eigenvalue} and \(w\) a corresponding \emph{Steklov eigenfunction}.
\end{definition}
It is clear from the definition that constant functions are eigenfunctions to the eigenvalue zero. The main reason for considering Steklov eigenfunctions is that they provide an orthonormal basis for the weakly harmonic functions $\mathcal{H}(\Omega) = \{ u\in H^1(\Omega)\mid -\Delta u = 0 \ \text{in }H^1_0(\Omega)^* \}$.
\begin{theorem}[Steklov spectral theorem]\label{steklovSpectralTheorem}\label{thm:SteklovSpectral}
    Let $\Omega$ be a bounded Lipschitz domain. Then there exists a non decreasing sequence $(\mu_j)_{j \in \mathbb{N}} \subseteq [0,\infty)$ with $\mu_j \to \infty$ and $(e_j)_{j \in \mathbb{N}} \subseteq \mathcal{H}(\Omega)$ such that $\mu_j$ is a Steklov eigenvalue with eigenfunction $e_j$. Further, $(e_j)_{j\in\mathbb{N}}$ is a complete orthonormal system in \(\mathcal H(\Omega)\) with respect to the inner product \(a(u,v) = \int_\Omega \nabla u\nabla v\mathrm dx + \int_{\partial\Omega} u v\mathrm ds\) on $H^1(\Omega)$.
\end{theorem}
This allows us to establish \eqref{eq:AbstractEigenExpansion} rigorously. 
\begin{lemma}
Assume that $f\in H^1(\Omega)^*$ and suppose $u_0$ satisfies a homogeneous Dirichlet problem as in \eqref{eq:Dirichlet} and $u_\lambda$ satisfies a homogeneous Robin problem as in \eqref{eq:Robin}, then $v_\lambda = u_0 - u_\lambda$ is a member of $\mathcal{H}(\Omega)$ and can be expanded in the form \eqref{eq:AbstractEigenExpansion}.
\end{lemma}
\begin{proof}
This follows directly from Theorem \ref{thm:SteklovSpectral}.
\end{proof}
\section{Integral Discretization}\label{appendix:IntegralDiscretization}
Here we discuss the influence of the approximation of the integrals in the loss function \eqref{eq:LossFunction}. As mentioned in remark \ref{remark:Warning} the discretized loss function is ill-posed on expressive sets of ansatz functions.
\begin{wrapfigure}{l}{0.54\linewidth}
    \includegraphics[width=0.95\linewidth]{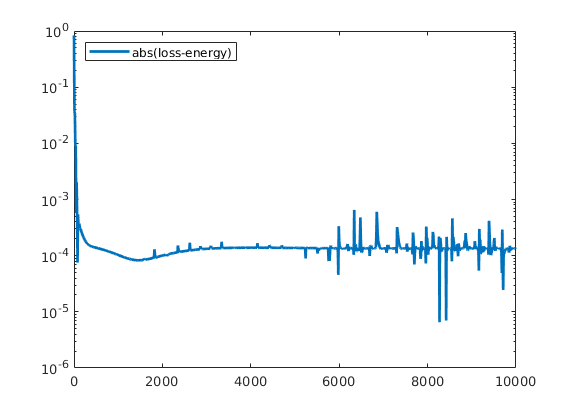}
    \caption{Evolution of the absolute value of the difference of the loss and the energy over the training process.}
    \label{pic:LossVsEnergy}
\end{wrapfigure}
To investigate whether this ill-posedness leads to relevant errors we consider the experiments on the disk, the annulus and the square as described in section \ref{sec:NumericalMethods}. 
For the disk, we conduct a training using roughly $90000$ evaluation points in the interior and $1000$ evaluation points on the boundary to approximate the integrals. More precisely, for the integral over $B_1(0)$ we use a $\varepsilon\cdot\mathbb{Z}^2$ grid with $\varepsilon = 1/160$. These were the values we employed in the other experiments throughout the article as well. Furthermore, we use a fixed penalization strength of $\lambda = 100$ and $10000$ iterations of the Adam optimizer with learning rate $0.001$. The remaining hyperparameters are chosen as reported in section \ref{sec:NumericalMethods}. To monitor the influence of the discretization, at every tenth step of the optimization process we compute the value of the loss function in two ways. On the one hand using the $90000$ evaluation points that are seen while training and on the other we use a million uniformly drawn evaluation points in the interior and $10000$ on the boundary to capture a possible deviation of the two. We refer to the latter as the ``energy'' in contrast to the loss used for training. In figure \ref{pic:LossVsEnergy}, we present the evolution of the two quantities over the training process.

We observe that the loss and the energy differ by a value of magnitude $10^{-4}$ apart from occasional oscillations. We attribute these oscillations to a comparatively large learning rate as the loss function itself does exhibit oscillations as well\footnote{For brevity, no figure is provided.}. Especially, no consistent increase of the energy-loss difference is visible in the late training process, where overfitting would be expected. We therefore conclude that the numerical approximation of the integrals does not significantly contribute to the error of the Deep Ritz Method in our experiments.

For the annulus and the square, we conduct analogous experiments. For these domains, we use roughly $250000$ evaluation points, i.e., we use a $\varepsilon\cdot\mathbb{Z}^2$ grid with $\varepsilon = 1/160$ for the annulus and $\varepsilon=1/500$ for the square. In the case of the annulus, the absolute value of the difference of the energy and the loss is around $10^{-3}$ and in the case of the square it is as large as $10^{-1}$. However, this large value for the square experiment can be explained by the fact that in the right-hand side of this equation the large factor $8\pi^2$ appears.
\end{document}